\documentclass[a4paper,11pt]{amsart}

\usepackage[utf8]{inputenc}
\usepackage[english]{babel}
\usepackage{amsfonts}
\usepackage{amsmath}
\usepackage{mathrsfs}
\usepackage{amssymb}
\usepackage{amsthm}
\usepackage{amscd}
\usepackage{latexsym}
\usepackage{amstext}
\usepackage{amsxtra}
\usepackage[alphabetic,backrefs,lite]{amsrefs} 
\usepackage{color}
\usepackage{paralist}
\usepackage[colorinlistoftodos]{todonotes}
\usepackage[all]{xy}

\usepackage{enumerate}
\usepackage{multirow}
\usepackage{wasysym}
\usepackage{latexsym} 
\usepackage{comment}
\usepackage{colonequals}

\usepackage[colorinlistoftodos]{todonotes}

\usepackage{geometry}
\usepackage[
        draft=false,
        colorlinks, citecolor=darkgreen,
        backref=page,
        pdfauthor={F. F. Favale, G. P. Pirola, D. Bricalli},
        pdftitle={GordanNoether},
        linktocpage        
]{hyperref}
\hypersetup{citecolor=blue,linktocpage}

\newtheorem*{thA}{Theorem A}
\newtheorem*{thB}{Theorem B}
\newtheorem*{thC}{Theorem C}

\newtheorem{theorem}{Theorem}[section]
\newtheorem*{theorem*}{Theorem}
\newtheorem{lemma}[theorem]{Lemma}
\newtheorem{corollary}[theorem]{Corollary}
\newtheorem{proposition}[theorem]{Proposition}
\newtheorem{remark}[theorem]{Remark}
\newtheorem{definition}[theorem]{Definition}
\newtheorem{example}[theorem]{Example}

\newcommand{\nc}{\newcommand} 
\nc{\cH}{{\mathcal H}}
\nc{\cA}{{\mathcal A}}
\nc{\cG}{{\mathcal G}}
\nc{\cC}{{\mathcal C}}
\nc{\cO}{{\mathcal O}}
\nc{\cI}{{\mathcal I}}
\nc{\cB}{{\mathcal B}}
\nc{\cY}{{\mathcal Y}}
\nc{\cK}{{\mathcal K}} 
\nc{\cX}{{\mathcal X}}
\nc{\cS}{{\mathcal S}}
\nc{\cE}{{\mathcal E}}
\nc{\cF}{{\mathcal F}}
\nc{\cZ}{{\mathcal Z}}
\nc{\cQ}{{\mathcal Q}}
\nc{\cN}{{\mathcal N}}
\nc{\cP}{{\mathcal P}}
\nc{\cL}{{\mathcal L}}
\nc{\cM}{{\mathcal M}}
\nc{\cT}{{\mathcal T}}
\nc{\cW}{{\mathcal W}}
\nc{\cU}{{\mathcal U}}
\nc{\cJ}{{\mathcal J}}
\nc{\cV}{{\mathcal V}}
\nc{\bH}{{\mathbb H}}
\nc{\bA}{{\mathbb A}}
\nc{\bG}{{\mathbb G}}
\nc{\bC}{{\mathbb C}}
\nc{\bO}{{\mathbb O}}
\nc{\bI}{{\mathbb I}}
\nc{\bB}{{\mathbb B}}
\nc{\bY}{{\mathbb Y}}
\nc{\bK}{{\mathbb K}} 
\nc{\bX}{{\mathbb X}}
\nc{\bS}{{\mathbb S}}
\nc{\bE}{{\mathbb E}}
\nc{\bF}{{\mathbb F}}
\nc{\bZ}{{\mathbb Z}}
\nc{\bQ}{{\mathbb Q}}
\nc{\bN}{{\mathbb N}}
\nc{\bP}{{\mathbb P}}
\nc{\bL}{{\mathbb L}}
\nc{\bM}{{\mathbb M}}
\nc{\bT}{{\mathbb T}}
\nc{\bW}{{\mathbb W}}
\nc{\bU}{{\mathbb U}}
\nc{\bD}{{\mathbb D}}
\nc{\bJ}{{\mathbb J}}
\nc{\bV}{{\mathbb V}}
\nc{\bbZ}{{\mathbb Z}}
\nc{\bR}{{\mathbb R}}
\nc{\fr}{{\rightarrow}}
\nc{\co}{{\nabla}}

\nc{\cu}{{\barline{\nabla}}}

\DeclareMathOperator{\Sec}{Sec}
\DeclareMathOperator{\Ann}{Ann}
\DeclareMathOperator{\Hess}{Hess}
\DeclareMathOperator{\hess}{hess}

\title[A theorem of Gordan and Noether via Gorenstein Rings]{A theorem of Gordan and Noether via Gorenstein Rings}

\author{Davide Bricalli}
\address{Dipartimento di Matematica,
	Universit\`a degli Studi di Pavia,
	Via Ferrata, 5
	I-27100 Pavia, Italy}
\email{d.bricalli1@campus.unimib.it}

\author{Filippo Francesco Favale}
\address{Dipartimento di Matematica,
	Universit\`a degli Studi di Pavia,
	Via Ferrata, 5
	I-27100 Pavia, Italy}
\email{filippo.favale@unipv.it}

\author{Gian Pietro Pirola}
\address{Dipartimento di Matematica,
	Universit\`a degli Studi di Pavia,
	Via Ferrata, 5
	I-27100 Pavia, Italy}
\email{gianpietro.pirola@unipv.it}

\date{\today}


\begin{document}

\maketitle



\begin{abstract}
Gordan and Noether proved in their fundamental theorem that an hypersurface $X=V(F)\subseteq \bP^n$ with $n\leq 3$ is a cone if and only if $F$ has vanishing hessian (i.e. the determinant of the Hessian matrix). They also showed that the statement is false if $n\geq 4$, by giving some counterexamples. 
Since their proof, several others have been proposed in the literature. In this paper we give a new one by using a different perspective which involves the study of standard Artinian Gorenstein $\bK$-algebras and the Lefschetz properties. As a further application of our setting, we prove that a standard Artinian Gorenstein algebra $R=\bK[x_0,\dots,x_4]/J$ with $J$ generated by a regular sequence of quadrics has the strong Lefschetz property. In particular, this holds for Jacobian rings associated to smooth cubic threefolds.
\end{abstract}

\section*{Introduction}
In a fundamental memoir of 1876, Gordan and Noether \cite{GN} fixed the statement of Hesse ({\cites{Hes51,Hes59}}) by showing that a complex projective hypersurface $V(F)\subset \bP^n$ with $n\leq 3$ is a cone if and only if the determinant of the Hessian of $F$ is zero. 
They also provided counterexamples for $n>3$. We can state the relevant part of this result as follows:

\begin{thA} 
Let $X=V(F)\subset \bP^n$ be a hypersurface defined over a field $\bK$ of characteristic $0$. 
Let $\Hess(F)$ be the hessian matrix of $F$ and assume $\hess(F)=\det(\Hess((F))\equiv 0$. 
Then, if $n\leq 3$, $X$ is a cone.
\end{thA}

The Gordan-Noether theorem has inspired and inspires many researchers (see, for example, the recent articles \cite{CO20,DS21}) and it has been revisited many times (see \cite{Los04,GR09,Wa20}) and the excellent last chapter of the book of Francesco Russo \cite{Rus16}. Moreover, via  Macaulay's inverse systems theory \cite{Mac94}, it has a surprising application to the theory of standard Artinian Gorenstein algebras (SAGAs, in short).

To explain it (but please see Section \ref{SEC:1} for details) we recall that a  $\bK$-graded algebra $R=\oplus_{i=0}^N R^i$ is a SAGA if, for all $i$, 
$\dim_{\bK}(R^i) < +\infty$, $R^0=\bK$,
$R$ is generated in degree $1$ and it satisfies the Poincar\'e duality (that is $\dim R^N=1$ and the pairing $R^i\times R^{N-i}\to R^N$ given by the multiplication is perfect).
The codimension of an Artinian algebra $R$ is, by definition, the dimension of the vector space $R^1$ of the elements of $R$ with degree $1$. 
Roughly speaking, $R$ has the structure of the even cohomology ring of an oriented compact variety $X$ of even dimension (generated in degree $2$). 
If $X$ is a K\"ahler variety of complex dimension $m$, the Hard Lefschetz Theorem (see \cite[Theorem 6.25, page 148]{textVoi1}) states that the cup product of the $r$-th power of a K\"ahler form induces an isomorphism between $H^{m-r}(X)$ and $H^{m+r}(X)$. A natural question is whether analogous properties, which in the literature are called (weak and strong) Lefschetz properties (see Definition \ref{DEF:LefProp} for details), hold for an Artinian Gorenstein algebra $R$.

It turns out that Gordan-Noether Theorem is then equivalent to the following (see \cite{Rus16,bookLef} and Section \ref{SEC:3} below for details):

\begin{thB} 
For all standard Artinian Gorenstein $\bK$-algebras $R$ of codimension at most $4$,
there exists $x\in R^1$ such that $x^{N-2}:R^1\to R^{N-1}$ is an isomorphism, i.e. the strong Lefschetz property holds in degree $1$.
\end{thB}

This equivalence is well known (it is proved in Section \ref{SEC:3} for completeness) and it follows from Macaulay's theory which allows to construct any SAGA starting from a homogeneous form in a finite number of variables (see Section \ref{SEC:1} for details or \cite[Theorem 2.71]{bookLef}).
\medskip

In this paper we reverse the logical line of the proof. We first give a direct proof of Theorem B and then we deduce Theorem A from this.  
To describe our approach we first remark that the statement of Theorem B is purely algebraic, but our proof is almost completely geometric and elementary. There are two main points that we would like to emphasize.
\medskip

The first one is the comparison between the Gorenstein duality and the projective duality theorem. This leads us to prove Proposition \ref{PROP:DUAL} that allows to treat directly the problem without introducing any auxiliar hypersurface (this was necessary in the original proof).
\medskip

The second point we want to highlight is the replacement of the famous Gordan-Noether identity. We assume that the Lefschetz property fails, that is the multiplication
by $x^{N-2}$ has non trivial kernel for $x\in R^1$ general. We construct then an incidence correspondence $\Gamma_{N-2}\subset \bP(R^1)\times \bP(R^1)$, where $\Gamma_{N-2}=\{([x],[y])\,|\, x^{N-2}y=0\}$ and it is such that its first projection is dominant. Our aim is to show that this implies that $\dim_{\bK}R^1>4$, i.e $\dim \bP(R^1)>3$.

Exploiting the differential condition that the kernel of $x^{N-2}$
must deform (we refer to this fact as the \emph{ker-coker} principle), we obtain a collection of equations for an irreducible component of $\Gamma_{N-2}$ (see Proposition \ref{PROP:CONE}). This is equivalent to what we call \emph{Gorenstein-Gordan-Noether identity} (see Corollary \ref{COR:GGN}).
\medskip

We recall that the Gordan-Noether identity is very important and it is the heart of the classical treatment of the Gordan-Noether Theorem and, as well, of all the proofs we have found in the literature.  The proof of the identity involves some delicate manipulations, and, in our opinion, as strong as it is, it appears as a cumberstone along the street of the proof. On the other hand, our Gorenstein-Gordan-Noether identity has a very elementary treatment and, as the original identity, it is a key relation for proving Theorem B. Moreover, in Subsection \ref{SUBSEC:GNI}, we show that our identity implies the relevant condition obtained from the Gordan-Noether identity when $\hess(F)=0$, used in the classical proof. 
\vspace{2mm}

We remark that the ker-coker principle has been used in \cite{FP21} for studying the Jacobian ring of a smooth plane curve in connection with the infinitesimal variation of the periods of the curve. It was in that article that we realized that these methods could be used in a more systematic way.
\bigskip

A natural question is whether these methods have more applications and in particular if they could be applied to study problems related to higher strong or weak Lefschetz properties for Gorenstein rings. As been observed for instance in \cites{MW09,Gon17}, all these properties are related to the vanishing of some higher hessians. Some higher Gorenstein-Gordan-Noether identities still hold but they give quite weaker information. 
\vspace{2mm}

The interest in the weak and strong Lefschetz properties for Artinian algebras has been developed in the last twenty years with important contributions by several authors. 
Just to mention some of them, the interested reader can see \cite{HMNW03,BK07,MN13,GZ18, Ila18, DGI20}. Much interest has been given to particular Artinian algebras, i.e. Jacobian rings of smooth hypersurfaces of degree $d$ in $\bP^n$ (see Example \ref{EX:SAGAviaREGSEQ}) over a field of characteristic $0$. 
If $R$ is one of such algebras then it has codimension $n+1$. If $n\leq 1$ then the weak and strong Lefschetz properties hold for $R$ as shown in the previously cited works. In the very recent article \cite{DGI20} it is proved that the Jacobian ring of a quartic curve or of a cubic surface has the strong Lefschetz property. In this article, we use our methods to analyse one of the first open cases and we show the following:
\begin{thC}
Let $S=\bK[x_0,\cdots,x_4]$ and let $I\subset S$ be an ideal generated by a regular sequence of length $5$ of quadrics. Then $R=S/I$ satisfies the strong Lefschetz property, i.e. the general element $x\in R^1$ is such that
$$x^3\cdot :R^1\to R^4 \qquad\mbox{ and }\qquad x\cdot :R^2\to R^3$$
are both isomorphisms. In particular, the Jacobian ring of a smooth cubic $3$-fold has the strong Lefschetz property.
\end{thC}

Moreover, as we prove that SLP holds for particular complete intersection SAGAs, our result also gives some evidence of the conjecture (\cite[Conjecture 3.46 pag. 120]{bookLef}) which claims that, in characteristic $0$, all complete intersection standard Artinian algebras should satisfy the weak and the strong Lefschetz properties. This is known, in the weak case, when the codimension is at most $3$ (see \cite{HMNW03}). It would be very interesting to know if all the Jacobian rings of smooth cubics satisfy some Lefschetz properties.
\vspace{2mm}

The article is divided into two parts. The first one comprises sections \ref{SEC:1},\ref{SEC:2} and \ref{SEC:3} and it is devoted to our proof of Gordan-Noether Theorem. More precisely, in Section \ref{SEC:1} we fix some notations and we recall some facts we will use later on. Section \ref{SEC:2} is the heart of the work where we prove Gorenstein-Gordan-Noether identity and Theorem B (see Theorem \ref{THM:LEFSCHETZ1}). In Section \ref{SEC:3}, we prove the Gordan-Noether theorem (Theorem A) by recalling some standard facts form Macaulay's theory, which show that it is equivalent to Theorem B. 
Moreover, we show that our Gorenstein-Gordan-Noether identity implies a relevant consequence of the classical identity. We have tried to keep this first part completely self-contained (up to Macaulay's theory) and to avoid using any unnecessary assumption. For instance we do not use that the vanishing of the hessian of $f\in \bK[x_0,\dots,x_n]$ does not depend on the reducedness of $f$, as shown by the beautiful result of Dimca and Papadima (see \cite{DP03}).
The second part is devoted to prove Theorem C (in Section \ref{SEC:4}) and to study a classical example with our methods. More precisely, in Section \ref{SEC:5}, we will analyze the intriguing components of the incidence correspondence $\Gamma_{N-2}$ of the easiest counterexample to Hesse's claim in $\bP^4$: the Perazzo (singular) cubic $3$-fold.
\bigskip

\noindent Acknowledgements: \\
The authors are thankful to Francesco Russo for helpful support on the topics of the paper and its observations. The authors are partially supported by INdAM - GNSAGA and by PRIN \emph{``Moduli spaces and Lie theory''} and by (MIUR): Dipartimenti di Eccellenza Program (2018-2022) - Dept. of Math. Univ. of Pavia.\\


\section{Artinian algebras and Lefschetz properties}
\label{SEC:1}

Good references for the content of this introductory section are \cite{textVoi2,bookLef,Rus16}. Let $\bK$ be a field of characteristic $0$ which is algebraically closed. In this paper we will deal with standard Artinian Gorenstein algebras:

\begin{definition}
\label{DEF:ARTGOR}
Let $R=\bigoplus_{i=0}^NR^i$ be an Artinian graded $\bK$-algebra. Then
\begin{itemize}
    \item the {\bf codimension} of $R$ is the dimension of $R^1$ as $\bK$-vector space; 
    \item $R$ is said to be {\bf standard} if it is generated, as $\bK$-algebra, by $R^1$;
    \item $R$ is said to have the {\bf Poincar{\'e} duality} if $R^N\simeq \bK$ and the multiplication map $R^s\times R^{N-s}\to R^{N}$ is a perfect pairing whenever $0\leq s\leq N$.
\end{itemize}
If $R$ is a graded Artinian algebra, having the Poincar{\'e} duality is equivalent to ask that $R$ is Gorenstein so the above duality is also called {\bf Gorenstein duality}. If $R$ is Gorenstein, $R^N$ is called {\bf socle} of $R$ and $R^0\simeq R^N$. 
\end{definition}

We recall two ways for constructing standard Artinian Gorenstein algebras (SAGAs in short) which are relevant for our work. We will denote by $S=\bK[x_0,\cdots,x_n]$ the polynomial ring in $n+1\geq1$ variables.

\begin{example}
\label{EX:SAGAviaREGSEQ}
If $e_0,\dots,e_n\geq 1$, consider a regular sequence $\{f_0,\dots,f_n\}$ in $S$ with $f_i\in S^{e_i}$. If $I=(f_0,\dots,f_n)$ then $S/I$ is a standard Artinian Gorenstein algebra with socle in degree $\sum_{i=0}^n(e_i-1)$. Particular algebras obtained via this construction are Jacobian rings associated to smooth hypersurfaces of degree $d\geq 2$ in $\bP^n$. In this case, if $X=V(F)$, one takes $f_i=\partial F/\partial x_i\in S^{d-1}$ so that $N=(d-2)(n+1)$ and $I$ is the ideal generated by the partial derivatives of $F$ (and it is called Jacobian ideal of $F$). 
\end{example}

\begin{example}
\label{EX:SAGAviaANN}
Consider the ring $\cQ$ of differential operators in the variables $x_0,\dots,x_n$, i.e. $\cQ=\bK\left[y_0,\cdots,y_n\right]$, where we have defined $y_i=\frac{\partial}{\partial x_i}$ for brevity. If $G\in \bK[x_0,\cdots,x_n]$ is any fixed homogeneous polynomial of degree $d\geq 1$, one can define the annihilator of $G$ in $\cQ$ as the ideal $$\Ann_{\cQ}(G)=\{D\in \cQ\,|\, D(G)=0\}.$$
It is not difficult to see that the quotient $R=\cQ/\Ann_{\cQ}(G)$ is a standard Artinian Gorenstein algebra with socle in degree $d$. 
\end{example}

All standard Artinian Gorenstein algebras have a description as in Example \ref{EX:SAGAviaANN} by an important result of Macaulay (see \cite{Mac94} for a revisited reprint of originary work by Macaulay of 1916, or also \cite[pag. 189]{Rus16} for the statement with modern language).

\begin{definition}
\label{DEF:LefProp}
Let $R=\bigoplus_{i=0}^NR^i$ be an Artinian Gorenstein graded $\bK$-algebra. We say that $R$ satisfies the 
\begin{itemize}
    \item weak Lefschetz property in degree $k$, ($WLP_k$ in short) if there exists $L\in R^1$ such that $L\cdot :R^k\to R^{k+1}$ has maximal rank;
    \item strong Lefschetz property (in narrow sense) in degree $k\leq N/2$, ($SLP_k$  in short) if there exists $L\in R^1$ such that $L^{N-2k}\cdot :R^k\to R^{N-k}$ is an isomorphism.
\end{itemize}
The algebra $R$ has the {\bf weak (strong) Lefschetz property} - $WLP$ (resp. $SLP$) in short - if it satisfies $WLP_k$ (resp. $SLP_k$) for all $k$.
\end{definition}

We recall a standard result concerning particular quotient Gorenstein rings that we will use in the following. For a simple proof one can refer to \cite{FP21}.

\begin{lemma} 
\label{LEM:RmodCOND}
Let $R=\oplus_{i=0}^N R^i$ be a Gorenstein ring with socle in degree $N$. Fix $\alpha\in R^e\setminus\{0\}$ and consider the ideal 
$$(0:\alpha)=\bigoplus_{i=0}^N\ker(\alpha\cdot :R^i\to R^{i+e}).$$
We have that $R_{\alpha}=R/(0:\alpha)$ is a Gorenstein ring with socle in degree $N_{\alpha}=N-e$. 
\end{lemma}


\section{SAGAs of codimension $\leq 4$ and strong Lefschetz property in degree $1$}
\label{SEC:2}

Let $\bK$ be an algebraically closed field with characteristic $0$. 
Let $R=\bigoplus_{i\ge 0}R^i$ a SAGA (i.e. a standard Artinian Gorenstein $\bK$-algebra) with socle in degree $N$ and assume $\dim R^1=n+1$ with $n\geq 1$.

For $2\leq i\leq N$ and $j\in\{1,\cdots,N-1\}$ we define 
$$Y_{i}= \{[y]\in \bP(R^1)\,|\, y^{i}=0\}\qquad \mbox{ and }\qquad 
\Gamma_j=\{([x],[y])\in \bP(R^1) \times \bP(R^1)\,|\, x^{j}y=0\}$$
where $\bP(R^1)$ is the projectivization of the vector space $R^1$.
By construction we have $Y_i\subseteq Y_{i+1}$ and $\Gamma_{j}\subseteq \Gamma_{j+1}$. 

\begin{remark}
\label{REM:HP}
Notice that, since $R$ is a standard $\bK$-algebra, if $i\leq N$, we must have $Y_i\subsetneq\bP(R^1)$. Indeed, if $Y_i=\bP(R^1)$ we would have that all $i$-th powers of elements of $R^1$ would be $0$. Since $R$ is standard, $i$-th powers of $R^1$ generate $R^i$ as vector space so $R^i=0$. This is clearly possible only if $i>N$.
\end{remark}

From now on, we will fix an integer $k$ such that $1\leq k\le N-2$ and, by denoting with $\pi_1$ and $\pi_2$ the projections of $\Gamma_k$ on the factors,
we will assume that the following condition holds:
$$(\star)\quad \pi_1:\Gamma_k\to \bP(R^1) \mbox{ is surjective}.$$

After proving some general results that hold for every value of $k$ in this range, in the following subsection we will focus on the specific case $k=N-2$, the most relevant one for our aims (see Remark \ref{REM:Valk}).

\begin{remark} 
\label{REM:Valk}
Notice that $(\star)$ is equivalent to asking that the multiplication map $x^k\cdot:R^1\to R^{k+1}$ is never injective, i.e. that $R$ does not satisfy $SLP$ at range $k$ in degree $1$. If $k=N-2$ (as we will assume from subsection \ref{SUBSEC:N-2} onwards), $(\star)$ holds if and only if $R$ does not satisfy $SLP_1$.
\end{remark}

Since $\pi_1:\Gamma_k\to \bP(R^1)$ is surjective, there exists an irreducible component of $\Gamma_k$ that dominates $\bP(R^1)$ via $\pi_1$. By observing that all the fibers of $\pi_1$ are irreducible (the fiber over $[x]$ is the projective space $[x]\times\bP(\Ann_{R^1}(x^k))$), one can easily obtain the following: 

\begin{lemma}
\label{LEM:UniqueGamma}
Under assumption $(\star)$, there exists a unique irreducible component of $\Gamma_k$ which dominates $\bP(R^1)$ via $\pi_1$.
\end{lemma}

We will denote by $\Gamma$ this unique component of $\Gamma_k$ that dominates $\bP(R^1)$ via $\pi_1$. For brevity, we will denote by $\pi_i$ also the restriction of $\pi_i$ to $\Gamma$ for $i=1,2$. Set $Y=\pi_2(\Gamma)$ and, for any $[y]\in Y$, 
$$F_y=\pi_1(\pi_2^{-1}([y])\cap \Gamma)=\{[x]\in \bP(R^1)\,|\, x^ky=0 \mbox{ and } ([x],[y])\in \Gamma\}.$$

The diagram summarizes the framework we are going to focus on for the whole article.

$$
\xymatrix@R=10pt{
F_y\times [y] \ar[ddd]_{\simeq} \ar@{^{(}->}[rd] \ar[rrr] & & & [y] \ar@{^{(}->}[d] \\ 
& \Gamma \ar@/_1.0pc/@{->>}[ddr]_-{\pi_1}  \ar@{->>}[rr]^-{\pi_2} \ar@{^{(}->}[rd]  & & Y \ar@{^{(}->}[d] \\
& & \Gamma_{k} \ar[r]^-{\pi_2} \ar@{->>}[d]^-{\pi_1} & \bP(R^1) \\
F_y \ar@{^{(}->}[rr] & & \bP(R^1)
}
$$


The following proposition gives a collection of equations which are satisfied by the points of $\Gamma$.

\begin{proposition} [Ker-Coker principle]
\label{PROP:CONE}
If $p=([x],[y])\in \Gamma$ then $$x^iy^j=0,$$ for all $i\geq 0$ and $j\geq 1$ such that $i+j=k+1$. 
\end{proposition}

\begin{proof}
Let us consider a general point $p=([x],[y])\in \Gamma$, so $x^ky=0$ by definition. We claim that $p$ satisfies also $x^{k-1}y^2=0$.

For any $v\in R^1$ and $t\in \bK$, let us take $x'=x+tv\in R^1$. Since $\pi_1:\Gamma\to \bP(R^1)$ is surjective by assumption $(\star)$, we have that there exists $y'$ in $R^1\setminus\{0\}$ such that $(x')^ky'=0$. Then we can define $\beta(t)$ such that $\beta(0)=y$ and $(x')^k\beta(t)=0$ for all $t\in \bK$. 
We can consider the expansion of $\beta$ and write this relation as
$$0\equiv(x+tv)^k(y+tw+t^2(\cdots))=x^ky+t(kvx^{k-1}y+wx^k)+t^2(\cdots).$$
If we multiply by $y$ we get that
$$kvx^{k-1}y^2=0 \qquad \forall v\in R^1.$$
Since the multiplication map $R^1\times R^{k+1}\to R^{k+2}$ is non degenerate we have that $x^{k-1}y^2=0$ as claimed. This proves that all the points of $\Gamma$ satisfy also the relation $x^{k-1}y^2=0$. 
\vspace{2mm}

In the same way one shows that if all the points of $\Gamma$ satisfy the relation $x^{i}y^{j}=0$ with $i+j=k+1$ and $j\geq 1$, then they also satisfy the relation $x^{i-1}y^{i+1}=0$. This concludes the proof.
\end{proof}

As an easy but fundamental consequence of Proposition \ref{PROP:CONE}, we obtain the following:
\begin{corollary}[Gorenstein-Gordan-Noether identity]
\label{COR:GGN}
Let $([x],[y])\in \Gamma$. Then the following relations hold for all $t\in \bK$ and $(\lambda:\mu)\in\bP^1$:
\begin{equation}
\label{EQ:GGN}
(x+ty)^{k+1}=x^{k+1}\in R^{k+1}\qquad\mbox{ and }\qquad [(\lambda x+\mu y)^{k+1}]=[x^{k+1}]\in \bP(R^{k+1})\quad (\mbox{if } x^{k+1}\neq 0).
\end{equation}
\end{corollary}

The second equality is simply the projective version of the first one. The origin of the name we have given to these equalities lies in the classical Gordan-Noether identity as explained in the introduction. In Section \ref{SEC:3}, by using the above relation, we will prove Equation \eqref{EQ:GN} that is one of the key identities used in the original proof of Gordan-Noether Theorem.

\begin{proposition}
\label{PROP:DIMY}
The following properties hold:
\begin{enumerate}[(a)]
\item $Y\subseteq Y_{k+1}=  \{[y]\in \bP : y^{k+1}=0\}\subsetneq \bP(R^1)$;
\item If $[y]\in Y$ is general, then $F_y$ is a cone with vertex $[y]$. Moreover, the general $F_y$ is connected;
\item $\dim F_y+\dim Y\geq \dim(\Gamma)\geq n$;
\item $1\leq \dim F_y\leq n-1$ and $1\leq \dim Y\leq n-1$.
\end{enumerate}
\end{proposition}

\begin{proof}
By Proposition \ref{PROP:CONE} we have that all the points of $\Gamma$ satisfy $y^{k+1}=0$. Then, by definition, we have $\pi_2(\Gamma)=Y\subseteq Y_{k+1}$. By Remark \ref{REM:HP} we have $Y_{k+1}\neq \bP(R^1)$ so we have proved claim $(a)$.
\vspace{2mm}

Before proving $(b)$, notice the following properties. For brevity, denote by $\Gamma^c$ the union of all the irreducible components of $\Gamma_k$ different from $\Gamma$. 
For any $p=([x],[y])\in \Gamma$ one can consider the curve $\gamma_p:\bP^1\to \bP(R^1)\times \bP(R^1)$ defined by 
$$\gamma_p((\lambda:\mu))=([\lambda x+\mu y], [y]).$$
Since $x^iy^j=0$ whenever $i+j=k+1$ and $j\geq 1$ we have $(\lambda x+\mu y)^ky=0$ so $\gamma_p$ has image in $\Gamma_k$. Whenever $p=([x],[y])\in \Gamma\setminus \Gamma^c$, we have that the curve $\gamma_p$ has image in $\Gamma$. In this case, the line parametrized by $\pi_1\circ\gamma_p$ is contained in $F_y$ and it is spanned by $[x]$ and $[y]$.
\vspace{2mm}

Now we will prove $(b)$. If $[y]\in Y$ is general, we have that $\pi_2^{-1}([y])\cap (\Gamma\setminus \Gamma^c)$ is an open dense subset of $F_y\times [y]$. Let $C$ be a connected component of $F_y$ and consider any $p=([x],[y])\in (C\times [y])\cap (\Gamma\setminus \Gamma^c)$, then the image of the curve $\gamma_p$ is contained in $\Gamma$ and pass through $([y],[y])$. So $[y]\in C$ and the line in $\bP(R^1)$ passing through $[x]$ and $[y]$ is contained in $C$. Since $[x]$ is general, we have that $C$ is a cone with vertex $[y]$. Moreover, if $C'$ is another connected component of $F_y$ we have $[y]\in C\cap C'$ so $C=C'=F_y$ and $F_y$ is connected.
\vspace{2mm}

In order to prove $(c)$ recall that $\Gamma$ and $Y$ are irreducible and $\pi_2:\Gamma\to Y$ is surjective. Then, for all $[y]\in Y$ we have
$$\dim(\pi_2^{-1}([y]))\geq \dim(\Gamma)-\dim(Y).$$
Since $\dim(\pi_2^{-1}([y]))=\dim F_y$ by definition of $F_y$ and since $\dim(\Gamma)\geq \dim(\bP(R^1))=n$ by hypothesis we get claim $(c)$.
\vspace{2mm}

For the last point $(d)$, fix $[y]\in Y$. Assume, by contradiction, that $\dim(F_y)=n$, i.e. $F_y=\bP(R^1)$. Then, for all $x\in R^1$, we have $x^ky=0$. Since $k$-th powers of elements in $R^1$ generates $R^k$ (since $R$ is a standard algebra) we have that $y\cdot R^k=0$. But this is impossible since $R$ is Gorenstein and $R^1\times R^k\to R^{k+1}$ is non-degenerate. This proves that $\dim(F_y)\leq n-1$. Using $(c)$ we also get that $\dim(Y)\geq 1$. By $(a)$ we have $\dim(Y)\leq \dim(Y_{k+1})<n$ so $\dim(Y)\leq n-1$. Then, using again $(c)$, we obtain $\dim(F_y)\geq 1$ as claimed. 
\end{proof}

In the next subsections our aim will be to improve the inequalities in point $(d)$ of Proposition \ref{PROP:DIMY} about the dimensions of $Y$ and of $F_y$ for $y$ general.


\subsection{The case $k=N-2$.}
\label{SUBSEC:N-2}
$\,$\\

From now on, we will assume $k=N-2$. Under this assumption, in this subsection, we prove a key result by constructing the dual variety of $Y$. As consequences, we will show that $Y$ is not linear and that its dimension is at most $n-2$. 
\vspace{2mm}

Consider the maps
$$\varphi:R^1\to R^{N-1} \quad x \mapsto x^{N-1}\quad \mbox{ and }\quad \psi:\bP(R^1)\setminus Y_{N-1}\to \bP(R^{N-1}) \quad [x] \mapsto [x^{N-1}].$$ 

Set $Z:=\overline{\psi(\bP(R^1)\setminus Y_{N-1})}$: this subvariety of $\bP(R^{N-1})$ is irreducible and non-degenerate (since $R$ is a standard algebra). For all $[x]\in\bP(R^1)$, let us define $K_x$ as the kernel of the differential
$$d_{[x]}\psi:T_{\bP(R^1),[x]}\rightarrow T_{Z,[x^{N-1}]}\qquad w\mapsto (d_{[x]}\psi)(w)=(N-1)x^{N-2}w$$
i.e. $K_x=\ker(d_{[x]}\psi)$. 
Set $\Delta$ to be the diagonal of $\bP(R^1)\times \bP(R^1)$. Whenever $p=([x],[y])\not \in \Delta$ we set 
\begin{equation}
L_p=\{[\lambda x+\mu y]\in \bP(R^1) \ | \ (\lambda:\mu)\in\bP^1\}.
\end{equation}
We have the following:

\begin{lemma}
For $p=([x],[y])\in \Gamma$ general, the line $L_p$ is contracted by $\psi$. Moreover, we have $\dim(Z)=n-\dim(K_x)\leq n-1$.
\end{lemma}
\begin{proof}
Let $p=([x],[y])\in \Gamma$ be general so we can assume that $x^{N-1}\neq 0$, i.e. $[x]\not\in Y_{N-1}$ (in particular, we have also $p\not \in \Delta$). Indeed, if we had $x^{N-1}=0$ for $p\in \Gamma$ general, then $\Gamma\subseteq Y_{N-1}\times Y$ and thus $Y_{N-1}=\bP(R^1)$ by $(\star)$. This is impossible by Proposition \ref{PROP:DIMY}(a). 

By using the Gorenstein-Gordan-Noether identity (see Corollary \ref{COR:GGN}) we have
$$\psi([\lambda x+\mu y])=[(\lambda x+\mu y)^{N-1}]=[\lambda^{N-1} x^{N-1}]=[x^{N-1}]=\psi([x])$$
so the line $L_p$ is contracted by $\psi$ (more precisely, $L_p\setminus Y_{N-1}$ is contracted to a point by $\psi$).
\vspace{2mm}

If we assume that $[z]:=\psi([x])=[x^{N-1}]\in Z_{smooth}$, then we have $\dim(Z)=\dim(T_{Z,[z]})$ so 
\begin{equation}
\label{EQ:DimZ}
\dim(Z)=\dim(\bP(R^{N-1}))-\dim(K_x)=n-\dim(K_x).
\end{equation}
Since $L_p$ is contracted by $\psi$ and $[x]\in L_p$ is not in $Y_{N-1}$ we have that $T_{L_p,[x]}=\langle y\rangle \subseteq K_x$ so $\dim(K_x)\geq 1$ and we have the claim.
\end{proof}

Recall that via Gorenstein duality we have a linear isomorphism $R^1\simeq (R^{N-1})^*$ which induces an isomorphism $\bP(R^{N-1})^*\simeq \bP((R^{N-1})^*)\simeq \bP(R^1)$. If $H\in \bP(R^{N-1})^*$ and $\alpha\in \bP((R^{N-1})^*)$ correspond under the first isomorphism, we have that the hyperplane $H$ contains a linear variety $\bP(W)\subseteq \bP(R^{N-1})$ if and only if $\alpha$, a linear form on $R^{N-1}$, annihilates all the vectors in $W$, i.e. $\alpha\in \bP(\Ann_{R^{1}}(W))$.
\vspace{2mm}

Let $X$ be a proper projective subvariety of $\bP^n$ and assume that $[x]\in X_{smooth}$. We will denote with $(\bP^n)^*$ the dual projective space of $\bP^n$ (i.e. the projective variety parametrizing the hyperplanes of $\bP^n$) and with $T_{[x]}(X)$ the projective tangent space to $[x]$ in $X$. If $\tilde{X}\subseteq \bK^{n+1}$ is the affine cone associated to $X$ we have $T_{[x]}(X)=\bP(T_{\tilde{X},x})$. We recall that the dual variety of $X$ (as subvariety of $\bP^n$) is
$$X^*=\overline{\{H\in (\bP^n)^*\quad |\quad \exists\, [x]\in X_{smooth}\, \mbox{ such that }\, T_{[x]}(X)\subseteq H\}}.$$

One of the key results of this subsection is the following:

\begin{proposition}
\label{PROP:DUAL}
We have $Y=Z^*$.
\end{proposition}
\begin{proof}
First of all, notice that if $[z]=[x^{N-1}]=\psi([x])\in Z_{smooth}$, we have that the tangent (projective) space to $Z$ in $[z]$ is described as
$$T_{[z]}(Z)=\bP(x^{N-2}\cdot R^1)=\bP\left(\{wx^{N-2}\,|\, w \in R^1\}\right).$$

Assume that $[y]\in Y$ is a general point. We claim that $y\in Z^*$. Since $[y]\in Y$ is general, we can take $([x],[y])\in \Gamma$ such that $x^{N-1}\neq 0$ and $[z]=[x^{N-1}]$ is a smooth point of $Z$. In particular $x^{N-2}y=0$ so $[y]\in \Ann_{R^1}(x^{N-2}\cdot R^1)$. Hence, by the above considerations, the hyperplane $H$ of $\bP(R^{N-1})$ corresponding to $[y]$ contains $T_{[z]}(Z)$ so $[y]\in Z^*$. Since $[y]$ was general in $Y$, we have proved $Y\subseteq Z^*$.
\vspace{2mm}

For the other inclusion, let $H$ be a general element in $Z^*$. Let $[y]\in \bP(R^1)$ be its corresponding point. Since $H\in Z^*$ (and $H$ is general) we have that there exists $[z]=[x^{N-1}]\in Z_{smooth}$ such that $H$ contains the tangent (projective) space $T_{[z]}(Z)$. Then, equivalently, $y$ annihilates $x^{N-2}\cdot R^1$. On the other hand, since the product $R^1\times R^{N-1}\to R^{N}$ is a perfect pairing, having $x^{N-2}wy=0$ for all $w\in R^1$ implies that $x^{N-2}y=0$ so $([x],[y])\in \Gamma_{N-2}$. Since $H$ was generic in $Z^*$ and, by Lemma \ref{LEM:UniqueGamma}, $\Gamma$ is the only component of $\Gamma_{N-2}$ which dominates $\bP(R^1)$ via $\pi_1$, we can assume that $[x]$ is outside $\overline{\pi_1(\Gamma_{N-2}\setminus \Gamma)}$. Then, we have that $([x],[y])\in \Gamma$ so $[y]\in Y$ as claimed.
\end{proof}

\begin{corollary}
\label{COR:NOLINEAR}
The variety $Y\subset \bP(R^1)$ cannot be linear.
\end{corollary}
\begin{proof}
Let us suppose by contradiction that $Y$ is a proper linear subvariety of $\bP(R^1)$. Since $\bK$ is a field of characteristic $0$, $Z$ is reflexive (see \cite{Wal56,Kle86}) so $Z=Z^{**}$. From Proposition \ref{PROP:DUAL}, we have that $Y=Z^*$ and so $Y^*=Z$.
Since we are assuming that $Y$ is linear, we have that also $Y^*$ is linear: namely, it is the linear subspace of $\bP(R^1)$ of the hyperplanes containing $Y$, which is proper. Then $Z$ is linear and thus degenerate, against the hypothesis.
\end{proof}


We are now going to improve the inequalities $(d)$ in Proposition \ref{PROP:DIMY}, by showing that $Y$ cannot be an hypersurface of $\bP(R^1)$ (see Corollary \ref{COR:DIMY2}). 
\vspace{2mm}

Let $p=([x],[y])\in \Gamma$ with $x\neq y$. As above, we will denote by $L_p$ the line in $\bP(R^1)$ joining the points $x$ and $y$, i.e. the line 
$L_p=\{\lambda x+\mu y\,|\, (\lambda:\mu)\in \bP^1\}$.

\begin{lemma} 
\label{LEM:LINE}
Let $p=([x],[y])\in \Gamma$ such that $x^{N-1}\neq 0$. Then
\begin{enumerate}[(a)]
\item $L_p\cap Y=[y]$;
\item if $p$ is general, $L_p$ is not tangent to $Y$ at $[y]$.
\end{enumerate}
\end{lemma}

\begin{proof}
Let $p=([x],[y])$ be as in the hypothesis and let us consider the line $L_p$. By Proposition \ref{PROP:CONE} we have that points in $Y$ satisfy $y^{N-1}=0$. Then
$$[y]\in L_p\cap Y \subseteq L_p\cap Y_{N-1}=\{[\lambda x+\mu y]\quad|\quad (\lambda:\mu)\in \bP^1,\, (\lambda x+\mu y)^{N-1}=0\}$$
On the other hand, the Gorenstein-Gordan-Noether identity (see Corollary \ref{COR:GGN}) yields $(\lambda x+\mu y)^{N-1}=\lambda^{N-1}x^{N-1}$. This is zero if and only if $\lambda=0$, so $L_p\cap Y=[y]$ as claimed.
\vspace{2mm}

Take $p=([x],[y])\in \Gamma$ general. Then, we can assume that $x^N\neq 0$ (since $R$ is a standard $\bK$-algebra - see Remark \ref{REM:HP}), that $p$ is a smooth point for $\Gamma$ and that the differential $d_p\pi_2:T_{\Gamma,p}\to T_{Y,[y]}$ is surjective.
\vspace{2mm}

Assume by contradiction that $L_p$ meets $Y$ non-transversely. Since the tangent in $[y]$ to $L_p$ is spanned by $x$ and since $d_p\pi_2$ is surjective, we have that there exists a tangent vector of the form $(v,x)$ in $T_{\Gamma,p}$. 
Hence, there is a curve $\gamma(t)$ in $\Gamma$, that we can write as $\gamma(t)=(\alpha(t),\beta(t))$, passing at $t=0$ through the point $p=([x],[y])$ and such that $\alpha'(0)=v$ and $\beta'(0)=x$.
As $\gamma$ has image in $\Gamma$, we have that $\alpha$ and $\beta$ satisfy the relation $\alpha(t)^{N-2}\beta(t)=0$. By considering the expansion of this relation as in Proposition \ref{PROP:CONE} we obtain the equation
$$(N-2)x^{N-3}vy+x^{N-1}=0.$$
If we multiply by $x$ we get $x^{N}=0$ which is impossible since we are assuming $x^N\neq 0$. Then $L_p$ and $Y$ meet transversely.
\end{proof}

\begin{corollary} 
\label{COR:DIMY2}
We have $$1\leq \dim(Y)\leq n-2\qquad \mbox{ and }\qquad 2\leq \dim(F_y)\leq n-1$$ for any $p=([x],[y])\in \Gamma$.
\end{corollary}

\begin{proof}
Assume by contradiction that $\dim(Y)=n-1$, so $Y$ is an irreducible hypersurface in $\bP(R^1)$. By Lemma \ref{LEM:LINE} there is a line which meets $Y$ transversely in one point, so $Y$ is an hyperplane. On the other hand, by Corollary \ref{COR:NOLINEAR}, $Y$ cannot be linear so we have a contradiction.
\end{proof}


\subsection{Proof of Theorem B}
$\,$\\

Our aim is now to prove the following:

\begin{proposition}
\label{PROP:MAININEQ}
If $\dim(Y)= 1$ then $n\geq 4$. 
\end{proposition}

Before doing so, we will need two technical results.

\begin{proposition}
\label{PROP:ALTERNATIVA}
Assume that $F_y$ has dimension $n-1$ for all $y\in Y$. Then
\begin{enumerate}[(a)]
\item $Y\subseteq \bigcap_{y\in Y}F_y$;
\item $\mathrm{Sec}(Y)\subseteq Y_{N-1}$.
\end{enumerate}
\end{proposition}

\begin{proof}
Recall that $\Gamma^c$ denotes the union of all the irreducible components of $\Gamma_{N-2}$ different from $\Gamma$.
Let us take an element $[y]\in Y$ and fix $[x]\in\bP(R^1)$ general, satisfying the following assumptions:
$$x^{N-1}\cdot y\ne 0,\qquad  [x]\in \pi_1(\Gamma\setminus \Gamma^c) \qquad \mbox{and} \qquad x^N\ne0.$$
This can be done since $R$ is a standard algebra and since $\Gamma$ is the only component dominating $\bP(R^1)$ via $\pi_1$.
\vspace{2mm}

Since $\pi_1$ is dominant, there exists $[y_1]\in Y$ (which can be assumed general as for $[x]$), such that $p_1=([x],[y_1])\in \Gamma\setminus \Gamma^c$. In particular, we have $x^{N-2}y_1=0$ and $[y_1]\ne[y]$ since, otherwise, we would have that $x^{N-2}y=x^{N-2}y_1=0$ which gives a contradiction.
\vspace{2mm}

Let us now consider the line $L_{p_1}$, joining the points $[x]$ and $[y_1]$, i.e. 
$$L_{p_1}=\{[\lambda y_1+\mu x]\quad|\quad (\lambda:\mu)\in \bP^1\}.$$ 
As in point $(b)$ of Proposition \ref{PROP:DIMY}, we have $L_{p_1}\subseteq F_{y_1}$ by the assumptions on $[x]$.

We claim now that $L_{p_1}\cap F_{y}={[y_1]}$.
\vspace{2mm}

Since, by assumption, $F_y$ has dimension $n-1$, the intersection $L_{p_1}\cap F_{y}$ cannot be empty. We will show now that $(L_{p_1}\setminus [y_1])\cap F_y$ is empty.
\vspace{2mm}

Notice that $L_{p_1}\setminus [y_1]$ is the affine line parametrized by $x(t)=x+ty_1$ with $t\in\bK$. Suppose that the intersection between $F_y$ and this affine line is not empty, i.e. there exists $\tilde{t}\in\bK$ such that $x+\tilde{t}y_1\in F_y$. This means that
$$(x+\tilde{t}y_1)^{N-2}y=0 \qquad \mbox{and multiplying by }x \qquad x(x+\tilde{t}y_1)^{N-2}y=0.$$
By construction, we have that $[x]\in F_{y_1}$ (equivalently, $([x],[y_1])\in\Gamma$), and so by Proposition \ref{PROP:CONE} we know that $x^iy_1^j=0$ for $j\ge1$ and $i+j=N-1$. Then we get $x(x+\tilde{t}y_1)^{N-2}=x^{N-1}$ and finally, by the above, $x^{N-1}y=0$, that is impossible by our assumptions. In conclusion, $L_{p_1}$ and $F_y$ meet each other at a single point, namely $[y_1]$.
\vspace{2mm}

We have proved that for general $[y_1]\in Y$ we have $[y_1]\in F_y$. Then, by the irreducibility of $Y$, we get $Y\subset F_y$. Since, this is true for every choice of $y\in Y$, we obtain claim $(a)$.
\vspace{2mm}

For $(b)$, let us consider two distinct points $[y_1],[y_2]\in Y$. From $(a)$ we have that $[y_2]\in F_{y_1}$ and then $p=([y_2],[y_1])\in \Gamma$. 
Let us now consider the projective line 
$$L_{p}=\{[\lambda y_1+\mu y_2] \,|\, (\lambda:\mu)\in\bP^1\}$$ so we have $L_{p}\subseteq \Sec(Y)$.
By Proposition \ref{PROP:CONE} we know that $y_2^iy_1^j=0$ for every $i,j$ with $j\ge1$ and $i+j=N-1$. On the other hand, we have that $y_2\in F_{y_2}$ so $0=y_2^{N-2}y_2=y_2^{N-1}$. By the above equations we get
$$(\lambda y_1+\mu y_2)^{N-1}=0$$
so $L_{p}\subseteq Y_{N-1}$.
Since every secant line is contained in $Y_{N-1}$, we have claim $(b)$.
\end{proof}

If we assume that $F_y$ has dimension $n-1$ for all $y\in Y$ we can strengthen the result of Corollary \ref{COR:DIMY2}:

\begin{proposition}
\label{PROP:DIMY3}
Assume that $F_y$ has dimension $n-1$ for all $y\in Y$. Then 
$1 \leq \dim(Y)\leq n-3$.
\end{proposition}

\begin{proof}
For all $[x]\in \bP(R^1)$, recall that $[x]\times \bP(K_x)$ is the fiber of $x$ with respect to $\pi_1:\Gamma\to \bP(R^1)$. Denote by $r-1$ the dimension of the general fiber $\bP(K_x)$.

$$
\xymatrix{
 & [x]\times \bP(K_x) \ar@{->>}[ld]\ar@{^{(}->}[rr] &  & \Gamma \ar@{->>}[dr]^{\pi_1}\ar@{->>}[dl]_{\pi_2} & & F_y\times [y]\ar@{_{(}->}[ll] \ar@{->>}[dr] \\
[x] \ar@{^{(}->}[rr] & & \bP(R^1) & & Y & & [y] \ar@{_{(}->}[ll]
}$$

Being $\bP(R^1)$, $\Gamma$ and $Y$ irreducible and $\pi_1$ and $\pi_2$ in the above diagram surjective by assumption, we have 
$$\dim(\Gamma)=\dim(\bP(R^1))+\dim(\bP(K_x))=n+r-1\qquad \dim(\Gamma)=\dim(Y)+\dim(F_y)=\dim(Y)+n-1$$
so $\dim(Y)=r$. By Corollary \ref{COR:DIMY2} we have $$1\leq \dim(Y)\leq n-2$$ so it is enough to prove that $\dim(Y)$ cannot be equal to $n-2$. This is clearly true if $n=2$ so we can assume $n\geq 3$. By contradiction, assume that $\dim(Y)=r=n-2$. Denote by $s$ the dimension of $\Sec(Y)$. By Proposition \ref{PROP:ALTERNATIVA} we have that $Y\subseteq \Sec(Y)\subseteq Y_{N-1}\subsetneq \bP(R^1)$ so we have $n-2\leq s\leq n-1$.
\vspace{2mm}

Notice, first of all, that $s$ cannot be $n-2$. Indeed, if $\dim(\Sec(Y))=\dim(Y)=n-2$, we would have that $Y$ is linear. This is impossible by Corollary \ref{COR:NOLINEAR}. Hence we can assume $s=n-1$.
\vspace{2mm}

Assume first that $Y$ is non-degenerate. We have that $Y$ and $\Sec(Y)$ have codimension $2$ and $1$ respectively in the smallest projective space that contains $Y$ (and $\Sec(Y)$). By considering the general hyperplane section $Y'$ and its secant variety $\Sec(Y')=\Sec(Y)\cap H$, we preserve the above properties and $Y'$ is as well non-degenerate. We can then cut with $n-3$ general hyperplanes in order to obtain a curve $C$ in $\bP^{3}$ and its secant variety which is a surface in $\bP^3$. This is impossible since, in this case, $C$ would be a plane curve.
\vspace{2mm}

The only remaining case to analyse is when $Y$ is degenerate of dimension $n-2$, $\dim (\Sec(Y))=n-1$ and the smallest projective subspace $H$ containing $Y$ is an hyperplane in $\bP(R^1)$. In particular, $Y$ is an hypersurface in $H=\Sec(Y)$ and its degree is at least $2$ (otherwise $Y$ would be linear).
\vspace{2mm}

First of all, we will prove that $H\subseteq F_y$ for $[y]\in Y$ general. Let $[y]\in Y$ be a general point. The general line $L$ through $[y]$ in $H$ cuts $Y$ in at least another point $[y_1]$. By Proposition \ref{PROP:ALTERNATIVA} $(a)$, we have that $[y],[y_1]\in F_{y}$ and then, by point $(b)$ of Proposition \ref{PROP:DIMY}, $L$ is contained in $F_y$. Since such lines cover a dense open subset of $H$ we have that $H\subseteq F_y$.
Then $H\times[y]\subset F_y\times [y]$ and  then $H\times Y\subseteq \Gamma$. Since they have the same dimension and they are both irreducible we have $H\times Y=\Gamma$. This is impossible by $(\star)$: if $H\times Y=\Gamma$ we would have $\pi_1(\Gamma)=H\neq \bP(R^1)$. Hence $\dim(Y)\leq n-3$ as claimed.
\end{proof}

We can now prove Proposition \ref{PROP:MAININEQ}: 
\begin{proof} Assume, by contradiction, that $n\leq 3$. Since we are assuming $\dim(Y)=1$ we have that $\dim(F_y)=n-1$ by Proposition \ref{PROP:DIMY}. Then, by Proposition \ref{PROP:DIMY3}, we have $1\leq \dim(Y)\leq n-3\leq 0$, which is clearly impossible.
\end{proof}

Let us now restate and prove our main result (Theorem B):

\begin{theorem}
\label{THM:LEFSCHETZ1}
For all standard Artinian Gorenstein $\bK$-algebras of codimension at most $4$
there exists $x\in R^1$ such that $x^{N-2}:R^1\to R^{N-1}$ is an isomorphism, i.e. the strong Lefschetz property holds in degree $1$.
\end{theorem}

\begin{proof}
Assume, by contradiction, that for all $x\in R^1$ the map $x^{N-2}\cdot :R^1\to R^{N-1}$ is not an isomorphism. Then, the projection $\pi_1:\Gamma_{N-2}\to \bP(R^1)$ is surjective, i.e. assumption $(\star)$ holds for $k=N-2$. Under these assumptions we have $1\leq \dim(Y)\leq n-2$ by Corollary \ref{COR:DIMY2}. Hence $n$ is equal to $3$ and $\dim(Y)=1$. This is impossible by Proposition \ref{PROP:MAININEQ}.
\end{proof}


\section{Gordan-Noether and strong Lefschetz property}
\label{SEC:3}

In this section we recall a well known result which shows that Theorem \ref{THM:LEFSCHETZ1} and the following one are equivalent. 

\begin{theorem}[Theorem A - Gordan-Noether]
\label{THM:GN}
Let $X=V(F)\subset \bP^n$ be a hypersurface with vanishing hessian (i.e. $\hess(F)=\det(\Hess((F))=0$). Then, if $n\le 3$, $X$ is a cone.
\end{theorem}

As a byproduct of this equivalence, our proof of Theorem \ref{THM:LEFSCHETZ1} gives a new proof of Gordan-Noether Theorem.
\vspace{2mm}

Let $\bK$ be a field and consider the polynomial ring in $m+1\geq1$ variables $S=\bK[x_0,\cdots,x_m]$ and the ring $\cQ$ of differential operators in such variables $\cQ=\bK\left[y_0,\cdots,y_m\right]$ with $y_i=\frac{\partial}{\partial x_i}$ as in Example \ref{EX:SAGAviaANN}. Let $\cA=Q/\Ann_Q(G)$ with $G\in S^d$ so that $\cA$ is a standard Artinian Gorenstein algebra with socle in degree $d$. Notice that the codimension of $\cA$, i.e. the dimension of $\cA^1$, is at most $m+1$ and equality holds as long as $(\Ann_{\cQ}(G))_1=\{0\}$. This is equivalent to ask that the partial derivatives of $G$ are linearly independent. Equivalently, $X=V(G)\subseteq \bP^m$ is not a cone.


\begin{lemma}
\label{LEM:EQUIVHESS0}
Fix $G\in S^d\setminus\{0\}$ and consider the SAGA $\cA=\cQ/\Ann_{\cQ}(G)$. Then $\cA$ has the strong Lefschetz property in degree $1$ if and only if $\hess(G)\not\equiv 0$. 
\end{lemma}

\begin{proof}
For any fixed $L=\sum_{i=0}^{m}k_i\frac{\partial}{\partial x_i}\in\cA^1$ we can consider the symmetric bilinear map
$$\varphi_L:\cA^1\times \cA^1\rightarrow \cA^{d}\simeq \bK$$
given by $\varphi_L(\eta,\xi)=(L^{d-2}\eta\xi)(G)$.
Let $\cB=\{y_0,\dots,y_m\}$ be a basis of $\cA^1$. Denote with $M_L$ the matrix associated to $\varphi_L$ with respect to $\cB$. Then we have $M_{L}=[\alpha_{ij}]_{0\leq i,j\leq m}$ with
$$\alpha_{ij}=(L^{d-2}y_iy_j)(G)=L^{d-2}(y_iy_j(G))=L^{d-2}(\Hess(G)_{ij})$$
where $\Hess(G)$ is the Hessian matrix of $G$. Since $\Hess(G)_{ij}$ is either 0 or has degree $d-2$, one can apply the differential Euler Identity (see \cite[Lemma 7.2.19]{Rus16}) in order to obtain
\begin{equation}
\label{EQ:ML}
M_L=(d-2)!\Hess(G)(k_0,\dots,k_m).
\end{equation}
Hence, having $\hess(G)\equiv 0$ is equivalent to ask that $\varphi_L$ is degenerate, i.e. for all $L,z\in \cA^1$ there exists $y\in \cA^1\setminus \{0\}$ such that $L^{d-2}yz=0$. By Gorenstein duality, this is equivalent to $L^{d-2}y=0$, i.e. $\cA$ does not satisfy the SLP in degree $1$.
\end{proof}

\begin{proposition}
Gordan-Noether Theorem (Theorem \ref{THM:GN}) is equivalent to Theorem \ref{THM:LEFSCHETZ1}.
\end{proposition}

\begin{proof}
{\bf Assume first that Theorem \ref{THM:LEFSCHETZ1} holds}. Let $X=V(F)$ be an hypersurface of degree $d\geq 2$ in $\bP^n$ with $n\leq 3$ and we assume that $X$ is not a cone. We have to show that $\hess(F)\neq 0$. Since $X$ is not a cone, the partial derivatives of $F$ are linearly independent. Hence, if we consider the SAGA $\cA=\cQ/\Ann_{\cQ}(F)$ as above, we have that $\cA$ has codimention $n+1\leq 4$ and socle in degree $d$. By Theorem \ref{THM:LEFSCHETZ1} $\cA$ has the strong Lefschetz property in degree $1$ so, by Lemma \ref{LEM:EQUIVHESS0}, $\hess(F)\neq 0$ as claimed.
\vspace{2mm}

{\bf Assume now that Theorem \ref{THM:GN} holds}. Let us consider a standard Artinian Gorenstein $\bK$-algebra $\cA$ of codimension $n+1$, with $n\le3$. This algebra can be described as
$$\cA=\frac{\cQ}{\Ann_{\cQ}(F)},$$
for a homogeneous polynomial $F$ of degree $d$ in the variables $x_0,\dots, x_n$ by Macaulay's Theorem. We can suppose that $F$ is such that $(\Ann_{\cQ}(F))_1=0$, i.e. its partial derivatives are linearly independent. Let us now assume by contradiction that $\cA$ does not satisfy the $SLP$ in degree $1$. Then, by Lemma \ref{LEM:EQUIVHESS0}, we would have $\hess(F)=0$. This is impossible since, by Theorem \ref{THM:GN} we would have that $V(F)$ is a cone: indeed, this would imply that the partial derivatives of $F$ are linearly dependent, which is against our assumptions. 
\end{proof}


\subsection{The Gordan-Noether identity}
\label{SUBSEC:GNI}
$\,$\\

We conclude this section by proving, using our framework and, in particular, the Gorenstein-Gordan-Noether identity \eqref{EQ:GGN}, an identity obtained by Gordan and Noether which is one of the key arguments in the original proof of Gordan-Noether Theorem. For brevity we called it Gordan-Noether identity. The original one is given in \cite[7.3.4]{Rus16}.
\medskip

First of all, let us introduce the formula (i.e. Equation \eqref{EQ:GN}) in its original setting. For any $h\in S=\bK[x_0,\dots,x_n]$ homogeneous, let $\nabla_h:\bP^n\dashrightarrow (\bP^n)^*$ be the polar map associated to $h$. Take $f\in S^d$ with $d\geq 1$ without multiple factors. 
The closure $Z'$ of the image of $\nabla_f$ in $(\bP^n)^*$ is easily seen to be a proper subvariety of $(\bP^n)^*$ if and only if $\hess(f)\equiv 0$. In this case, for any hypersurface $T=V(g)$ containing $Z'$, we can consider the \emph{Gordan-Noether map associated to $g$} which is 
$$\psi_g:=\nabla_g\circ \nabla_f :\bP^n\dashrightarrow\bP^n.$$
One of the key steps in the classical proof of Gordan-Noether Theorem is the following claim: if $f$ has vanishing hessian, then the {\it Gordan-Noether identity} \begin{equation}
\label{EQ:GN}
\psi_g(\underline{x}+\lambda\psi_g(\underline{x}))=\psi_g(\underline{x})
\end{equation}
holds for all $\lambda\in \bK$.
\vspace{2mm}

Let us now express the map $\psi_g$ using the framework introduced in Section \ref{SEC:2}. 
Let $R$ be a standard Artinian Gorenstein algebra with socle in degree $N$ and assume that the $SLP_1$ does not hold. By Macaulay Theorem we have $R=\cQ/\Ann_{\cQ}(F)$ for some suitable $F\in \bK[x_0,\cdots, x_n]$ with $\Ann_\cQ(F)_1=(0)$ and $\hess(F)=0$. By \cite[Lemma 3.74]{bookLef}, $F$ can be taken to be the function $x\mapsto x^N$ via the isomorphism $R^N\simeq \bK$.
\vspace{2mm}

By considering this function, one can observe that $\nabla_F$ is exactly the map $\psi:\bP(R^1)\dashrightarrow \bP(R^{N-1})$ introduced in Section \ref{SEC:2}, i.e. the map such that $\psi([x])=[x^{N-1}]$ for $[x]\in \bP(R^1)\setminus\{Y_{N-1}\}$, so our variety $Z$ coincides with the variety $Z'$ introduced above. If $T=V(g)$ is an hypersurface containing $Z$, then the Gordan-Noether map $\psi_g$ defined above is the composition $\nabla_g\circ \nabla_F=\nabla_g\circ \psi$. Since the image of $\nabla_g$ lives in $\bP(R^{N-1})^*\simeq \bP(R^1)$ we interpret $\psi_g$ as a (rational) map from $\bP(R^1)$ to $\bP(R^1)$. As $V(g)^*\subseteq Z^*=Y$, we have that, for all $g$ as above, the image of $\psi_g$ is contained in $Y$.

\begin{proposition}
The Gordan-Noether identity - Equation \eqref{EQ:GN} - follows from the Gorenstein-Gordan-Noether identity - Equation \eqref{EQ:GGN}.
\end{proposition}

\begin{proof}
Let $[x]\in \bP(R^1)$ be a general point. We can assume then, that $\psi([x])=[x^{N-1}]=[z]$ is a smooth point for $Z$. Set $[y]=\psi_g([x])$ and notice that $[y]\in Y$ as we have seen that the image of $\psi_g$ lies in $Y$. We claim that $([x],[y])\in \Gamma$. Indeed, $[y]\in Y\subset \bP(R^1)$ corresponds to an hyperplane $H_y$ of $\bP(R^{N-1})$ tangent to $V(g)$ containing the (projective) tangent space $T_{[z]}(Z)=\bP(x^{N-2}\cdot R^1)$ by construction. This implies that $y$ annihilates the vector space $x^{N-2}\cdot R^1$. Since $(x^{N-2}y)\cdot R^1=0$, by Gorenstein duality we have $x^{N-2}y=0$ so $([x],[y])\in \Gamma_{N-2}$. Since $[x]$ was general and since $\Gamma$ is the only component of $\Gamma_{N-2}$ dominating $\bP(R^1)$ via $\pi_1$, we have that $([x],[y])\in \Gamma$. Then, by the Gorenstein-Gordan-Noether identity (i.e. Equation \eqref{EQ:GGN}) we have
$$\psi_g([x]+\lambda\psi_g([x]))=\psi_g([x+\lambda y])=\nabla_g(\psi([x+\lambda y]))=\nabla_g(\psi([x]))=\psi_g([x])$$
as claimed.
\end{proof}

If $Z$ is an hypersurface we have a really simple description for the unique Gordan-Noether map in our framework.

\begin{remark}
\label{REM:ALPHAandGN}
Assume that $Z$ is an hypersurface. In this case there is only one Gordan-Noether map, i.e. the one associated to a generator $g$ for the ideal of $Z$, and one can see that it can be described as the dominant rational map $\alpha:\bP(R^1)\dashrightarrow
Y$ such that, for $[x]$ in a suitable dense open subset $U$, $\alpha([x])= \bP(\Ann_{R^1}(x^{N-2}))=[y]$, i.e. $\alpha([x])=[y]$ where $([x],[y])$ is the only point over $[x]$.
$$
\xymatrix{
\Gamma \ar@{->>}[r]^{\pi_2}\ar@/_1.0pc/@{->>}[d]^{\pi_1} & Y \\
\bP(R^1) \ar@/_1.0pc/@{-->}[u]\ar@{-->}[r]_-{\psi} \ar@{-->}[ru]_-{\psi_g}^-{\alpha}
& Z=V(g) \ar@{-->}[u]_{\nabla_g} \ar@{^{(}->}[r] & \bP(R^{N-1})
}
$$
\end{remark}


\section{Lefschetz properties for complete intersection SAGAs presented by quadrics}
\label{SEC:4}

Set $S=\bK[x_0,\dots,x_n]=\sum_i S^i$ and let $d\geq 3$. Assume that $I$ is an ideal of $S$ generated by $\{f_0,\dots,f_n\}$ with $f_i\in I^{d-1}=I\cap S^{d-1}$ for all $i$ and such that $\{f_0,\dots,f_n\}$ is a regular sequence.  In this case $R=S/I$ is a SAGA with socle in degree $N=(d-2)(n+1)$. In particular, $I^{d-1}$ is a vector space of dimension $n+1$ and, by Bertini, the general form in $I^{d-1}$ is smooth and irreducible. Examples of such rings are the Jacobian rings of smooth hypersurfaces of degree $d$ in $\bP^n$. 

For any $[\eta]\in \bP(R^h)$ we set 
$$K^i_\eta =\ker\left(R^i\stackrel{\cdot\eta}\to R^{i+h}\right).$$
 
\begin{proposition}
\label{PROP:SameKer}
Assume that $1 \leq h\leq N-1$. The following properties hold:
\begin{enumerate}[(a)]
\item If $\eta\in R^h\setminus \{0\}$, then $h\geq(d-2)\dim(K^1_{\eta})$;
\item Let $\eta,\zeta\in R^h\setminus \{0\}$ and assume $h=(d-2)\dim(K^1_{\eta})=(d-2)\dim(K^1_{\zeta})$. Then
$K^1_\eta=K^1_\zeta$ if and only if $[\eta]=[\zeta]$ in $\bP(R^h)$.
\end{enumerate}
\end{proposition}
\begin{proof}
Assume that $\dim (K^1_\eta)=k$ and chose $y_0\dots y_{k-1}$ linearly independent elements in $K^1_{\eta}$. We can find $g_k,\dots, g_{n}\in I^{d-1}$ such that $\tilde{I}=(y_0\dots y_{k-1},g_k,\dots, g_{n})$ is the irrelevant ideal (i.e. $\{y_0\dots y_{k-1},g_k,\dots, g_{n}\}$ is a regular sequence). Then $\tilde{R}=S/\tilde{I}$ is a standard Gorenstein Artinian algebra with socle in degree
$\tilde{N}=(d-2)(n+1-k)$. In particular, any element of $S$ of degree at least $\tilde{N}+1$ belongs to $\tilde{I}$. We claim that $\eta\cdot R^{\tilde{N}+1}=0$. Indeed, if $g\in S^{\tilde{N}+1}$ we have
$$\eta\cdot g=\eta\cdot \left(\sum_{i=0}^{k-1}\lambda_iy_i+\sum_{i=k}^{n}\mu_ig_i\right)\in I$$
since $y_i\in K^1_{\eta}$ and $g_i\in I$. 
This is possible, by Gorenstein duality, if and only if $\tilde{N}+h+1>N$, i.e. if and only if $h\geq (d-2)k$ as claimed by $(a)$.

\bigskip
For $(b)$ assume that $\eta,\zeta\in R^h\setminus\{0\}$ are such that $K^{1}_\eta=K^1_{\zeta}$ and $h=(d-2)\dim(K^1_{\eta})$. Then we can proceed as before and construct the ideal $\tilde{I}$ and the ring $\tilde{R}$ with socle in degree $\tilde{N}=N-h$.
We claim that $K^{N-h}_{\eta}=K^{N-h}_{\zeta}$. Let $\tilde{\sigma}$ be a representant of the socle of $\tilde{R}$. Then we can write $S^{N-h}=\langle\tilde{\sigma},\tilde{I}^{\tilde{N}}\rangle$. One can easily check that $\eta\cdot \tilde{I}\subseteq I$ and $\zeta\cdot \tilde{I}\subseteq I$. On the other hand, $\eta,\zeta$ are not zero so $\eta\cdot R^{N-h}$ and $\zeta\cdot R^{N-h}$ are not $0$, i.e. $\eta\cdot \tilde{\sigma},\zeta\cdot \tilde{\sigma}\neq 0$ in $R$. Hence, we have that $K^1_{\eta}=K^1_{\zeta}=\tilde{I}^{\tilde{N}}$ and then $\eta$ and $\zeta$ are multiples.
\end{proof}

As an application, we have the following bound for the dimension of $Y_a=\{[y]\in \bP(R^1)\,|\, y^a=0\}$.

\begin{corollary} 
\label{COR:DimYa}
We have $$\dim (Y_a) \leq \frac {a-1}{d-2}-1.$$
In particular, if $d=3$, we have
$\dim (Y_{a})\leq a-2$.
 
\end{corollary}
\begin{proof}
Take the general point
$[y]$ of any irreducible component $C$ of $Y_a$ of maximal dimension which is not contained in $Y_{a-1}$. If such a component does not exist, set $\epsilon>0$ to be the biggest integer such that $Y_a=Y_{a-\epsilon}$. The bound for $\dim(Y_{a-\epsilon})$ implies the one for the dimension of $Y_a$.

Let $\tilde{C}$ be the associated affine cone. We claim that $T_{\tilde{C},y}=K^1_{y^{a-1}}$. Indeed if $v$ is a tangent vector to $\tilde{C}$ in $y$, we have a curve $\gamma(t)=y+tv+t^2(\cdots)$ which is contained in $Y_a$. Then, by expanding the relation $\gamma(t)^a=0$, one has $vy^{a-1}=0$ so $v\in T_{\tilde{C},y}$ if and only if $v\in K^1_{y^{a-1}}$. Then, by Proposition \ref{PROP:SameKer}, we have
$$\dim(Y_a)=\dim(\tilde{C})-1=\dim(K^{1}_{y^{a-1}})-1\leq\frac{a-1}{d-2}-1$$
as claimed.
\end{proof}

As a consequence, we have a new proof of the following result of Migliore and Nagel (\cite[Proposition 4.3]{MN13}).

\begin{corollary}
\label{COR:WLP1}
Let $R=S/I$ be a standard Artinian Gorenstein algebra with $I$ generated by a regular sequence of polynomials of degree $e$ with $e\geq 2$. Then $R$ has the weak Lefschetz property in degree $1$.
\end{corollary}

\begin{proof}
The result is clear if $e\geq 3$ since, in this case, $R^1=S^1$ and $R^2=S^2$. If $e=2$ one can consider the incidence variety $\Gamma_{1}=\{([x],[y])\in \bP(R^1)\times \bP(R^1)\,|\, xy=0\}$ introduced in Section \ref{SEC:2} and its projection $\pi_1$ on $\bP(R^1)$. By contradiction, assume that the weak Lefschetz property does not hold in degree $1$. This is equivalent to ask that $\pi_1$ is surjective. Proceeding as in Section \ref{SEC:2} one has that there exists a unique irreducible component $\Gamma$ of $\Gamma_1$ that dominates $\bP(R^1)$ via $\pi_1$. Moreover we have $Y=\pi_2(\Gamma)\subseteq Y_2$ and $\dim(Y)\geq 1$ (proceeding as in Proposition \ref{PROP:DIMY}) so $\dim(Y_2)\geq 1$. On the other hand, by Corollary \ref{COR:DimYa} we have $\dim(Y_2)\leq 0$, which gives a contradiction.
\end{proof}


\subsection{Proof of Theorem C}
$\,$\\

\noindent In this subsection we will focus on the case $n=4$ and $d=3$, i.e. we will deal with standard Artinian Gorenstein algebras which are quotients of $S=\bK[x_0,\cdots,x_4]$ by ideals generated by regular sequences of length $5$ whose elements have degree $2$. We will prove that they always satisfy the strong Lefschetz property (in any degree). Under these assumptions we have $I=(I^2)$, $N=5$ and 
$$R=S/I=R^0\oplus R^1\oplus R^2\oplus R^3\oplus R^4\oplus R^5$$ with $(\dim(R^i))_{i=0}^5=(1,5,10,10,5,1)$. 
For simplicity, if $\alpha\in R^c$, we will define by $\mu_i(\alpha)$ to be the multiplication map by $\alpha$ from $R^i$ to $R^{i+c}$. In particular we have $K^i_{\alpha}=\ker(\mu_i(\alpha))$.

We will need the following technical result:

\begin{proposition} 
\label{PROP:techprop}
Let $[x]\in \bP(R^1) $ and $[q]\in \bP(R^2)$ such that $qx=0$. Let $W\subset K^2_q$ be a subspace with $\dim(W)\geq 4$. Then $W\cap (x\cdot R^1)\neq \{0\}$. 
\end{proposition}

\begin{proof} 
Consider the quotient $R_q=R/(0:q)$, i.e. the quotient of $R$ by the ideal $J$ such that $J^i=K^i_q$. This is a SAGA with socle in degree $N_q=N-\deg(q)=3$ by Lemma \ref{LEM:RmodCOND}. Since $xq=0$ by hypothesis, we have $K^1_q\neq 0$. By Proposition \ref{PROP:SameKer} we have $\dim(K^1_q)\leq 2$ and so $\dim(R_q^1)\in \{3,4\}$.
Since $\dim(K^2_q)=\dim(R^2)-\dim(R^2_q)=10-\dim(R^1_q)$
we have that $\dim(K^2_q)\in \{6,7\}$. In particular, $\dim(K^2_q)\leq 7$.
Consider $W\subseteq K^2_q$ of dimension $4$ and the subspace $V=x\cdot R^1$. By Proposition \ref{PROP:SameKer} we have that $V$ has dimension at least $5-1=4$ and, by construction, is a subspace of $K^2_q$. Then $W\cap V$ has dimension at least $1$ as claimed.
\end{proof}

We can now prove the main result of this section.

\begin{theorem}
Let $R$ be as above. Then $R$ satisfies the strong Lefschetz property, i.e. the general element $x\in R^1$ is such that
\begin{description}
\item [$SLP_1$] $\mu_1(x^3)=x^3\cdot :R^1\to R^4$
\item [$SLP_2$] $\mu_2(x)=x\cdot :R^2\to R^3$
\end{description}
are both isomorphisms.
\end{theorem}

\begin{proof}
The proof is organized in two steps: first of all we will prove that $SLP_1$ implies $SLP_2$ and then that $SLP_1$ holds.
\vspace{2mm}

{\bf Step 1:} $SLP_1 \Longrightarrow SLP_2$. We will proceed by contradiction by assuming that $SLP_2$ is false, i.e. that for all $x\in R^1$ we have $K^2_{x}\neq \{0\}$. We can consider the incidence variety 
$$\Gamma^{1,2}=\{([x],[q])\in \bP(R^1)\times \bP(R^2)\,|\, xq=0\}$$ 
and its projections $\pi_1$ and $\pi_2$. Since $SLP_2$ does not hold we have that $\pi_1$ is dominant.

If $[q]\in \bP(R^2)$, we have $\pi_2^{-1}([q])=\bP(K^1_q)\times [q]$ so its dimension is at most $1$ by Proposition \ref{PROP:SameKer}. As in Section \ref{SEC:2} one has that there exists a unique irreducible component $\Gamma$ of $\Gamma^{1,2}$ that dominates $\bP(R^1)$ via $\pi_1$. Let $Y$ be the image of $\Gamma$ via $\pi_2$.
Since $\pi_1$ is dominant we have that $\Gamma$ has dimension at least $4$. Then, for $[q]\in Y$ general,
$$\dim(Y)=\dim(\Gamma)-\dim(\Gamma\cap \pi_2^{-1}([q]))\geq 4-1=3.$$

We claim that $Y\subseteq Y_2^{(2)}:=\{[q]\in \bP(R^2)\,|\, q^2=0\}$. Let $([x],[q])$ be a generic point in $\Gamma$. Proceeding as in Proposition \ref{PROP:CONE}, since $\pi_1:\Gamma\to \bP(R^1)$ is dominant, for any $v\in R^1$, we can find $\beta(t)=q+tw+t^2(\cdots)\in Y\subset \bP(R^2)$ such that $(x+tv)\beta(t)=0$. Then, by considering the expansion of this relation modulo $t^2$ we obtain 
\begin{equation}
\label{EQ:KC1}
xw+qv=0 \mbox{ for all } v\in R^1.
\end{equation}
Then, by multiplying by $q$, one gets $q^2v=0$ for all $v\in R^1$. By Gorenstein duality we have $q^2=0$ so $Y\subseteq Y_2^{(2)}$ as claimed.

Since $p=([x],[q])$ was general we can also assume that $[q]$ is smooth for $Y$ and $Y^{(2)}_2$ and that the differential $d_p\pi_2:T_{\Gamma,p}\to T_{Y,q}$ is surjective. Then, as in Corollary \ref{COR:DimYa}, one can show that the Zarisky tangent space to the affine cone $\tilde{Y}_2^{(2)}$ of $Y_2^{(2)}$ at $q$ is $T_{\tilde{Y}_2^{(2)},q}\simeq K^2_{q}$. Since $\dim(Y)\geq 3$ and $Y\subseteq Y_2^{(2)}$ we can find three tangent vectors $w_1,w_2,w_3$ such that $W=\langle w_1,w_2,w_3,q\rangle$ is a $4$-dimensional subspace of $T_{\tilde{Y},q}\subseteq K^2_q$. Here $\tilde{Y}$ is the affine cone of $Y$. 
Then, by Proposition \ref{PROP:techprop}, we have $W\cap (x\cdot R^1)\neq \{0\}$ so we can find $\eta\in R^1\setminus \{0\}$ such that $x\eta\in W$. Notice that $x\eta$ cannot be equal to $q$ since, otherwise, we would have that for $[x]\in \bP(R^1)$ general $0=xq=x^2\eta$ and then $x^3\eta=0$: this is impossible since we are assuming $SLP_1$. Then $x\eta$ is not $0$ as tangent vector in $T_{Y,q}$. By the surjectivity of the differential map $d_p\pi_2$, there exists $v\in R^1$ such that 
$$(v,x\eta)\in T_{\Gamma,p}\subseteq T_{\bP(R^1),[x]}\times T_{\bP(R^2),[q]}$$ 
so there is a curve $(\alpha(t),\beta(t))\subseteq \Gamma$ passing through $p$ with tangent vector $(v,x\eta)$. By expanding at the first order one gets a relation as the one in Equation \eqref{EQ:KC1}: $x^2\eta+qv=0$. If we multiply by $x$ we have $x^3\eta=0$. However, this is only possible for $x$ special since we are assuming $SLP_1$ and so it leads to a contradiction.
\vspace{2mm}

{\bf Step 2}: $SLP_1$ holds. 
Assume, by contradiction, that $SLP_1$ does not hold. Then, as in Section \ref{SEC:2} we can construct $\Gamma_{N-2}=\Gamma_3=\{([x],[y])\in \bP(R^1)\times \bP(R^1)\,|\, x^3y=0\}$ which dominates $\bP(R^1)$ via its projection $\pi_1$. Let us consider $\Gamma\subseteq\Gamma_3$, the unique irreducible component that dominates $\bP(R^1)$ via $\pi_1$, and $Y=\pi_2(\Gamma)$. By Corollary \ref{COR:DIMY2} we have that $\dim(Y)\in \{1,2\}$.

By Macaulay's Theorem we have that $R\simeq Q/\Ann_Q(G)$ where $Q=\bK[y_0,\dots, y_4]$, with $y_i=\partial/\partial x_i$, $G\in S^N=S^5$ and $(\Ann_{Q}(G))_1=(0)$. Since $\dim(R^1)=5$, $R$ has codimension $5$ and, by assumption, $V(G)\subseteq \bP^4$ is not a cone. Moreover, by Lemma \ref{LEM:EQUIVHESS0}, $\hess(G)\equiv 0$. Then, by \cite[Lemma 7.4.13]{Rus16}, we have that $\dim(Y)$ cannot be $2$.
\vspace{2mm}

Assume then that $\dim(Y)=1$. By Proposition \ref{PROP:DIMY} we have that for all $[y]\in Y$, $F_y=\pi_2^{-1}([y])\cap \Gamma$ has dimension $3$. Let $p=([x],[y])$ be a general point of $\Gamma$. We may assume that $[x]$ is a smooth point of $F_y$. Then, if $\tilde{F}_y$ is the affine cone associated to $F_y$, we have $T_{\tilde{F}_y,x}\subseteq K^1_{x^2y}$. This can be proved by considering a curve in $\tilde{F}_y$ passing through $x$ with tangent $v$ as we have done before. 
\vspace{2mm}

We claim that $x^2y=0$. Assume by contradiction that $x^2y\neq 0$. Since $\tilde{F}_y$ has dimension $4$, we can find independent vectors $x,z_1,z_2,z_3 \in T_{\tilde{Y},y}$, such that $\langle x,z_1,z_2,z_3\rangle$ is a $4$-dimensional subspace of $K^1_{x^2y}$. On the other hand since $x^2y\neq 0$, by Proposition \ref{PROP:SameKer}, we have that $K^1_{x^2y}$ has dimension at most $3$. Hence, we get $x^2y=0$ as claimed. In particular, all the points of $\Gamma$ satisfy the relation $x^2y=0$: $\Gamma\subset \Gamma_2=\{([x],[y])\in \bP(R^1)\times \bP(R^1)\,|\, x^2y=0\}\subseteq \Gamma_3$. We have then that $\pi_1:\Gamma_2\to \bP(R^1)$ is dominant and $\Gamma$ is also an irreducible component of $\Gamma_2$ (which dominates $\bP(R^1)$ via $\pi_1$). Then $F_y\times [y]\subset \Gamma_2$ is the fiber over $[y]\in \bP(R^1)$ of $\pi_2:\Gamma\subseteq \Gamma_2\to \bP(R^1)$. Then we have that for $p=([x],[y])\in \Gamma$, $T_{\tilde{F}_y,x}\subseteq K^1_{xy}$. By proceeding as before one can prove that also the relation $xy=0$ holds for the points of $\Gamma$. Then, since the weak Lefschetz property holds in degree $1$ for $R$ by Corollary \ref{COR:WLP1}, $x$ is not general and we have a contradiction.
\end{proof}

We have the following important consequence:
\begin{corollary}
The Jacobian ring of a smooth cubic threefold has the strong Lefschetz property.
\end{corollary}


\section{The Perazzo cubic}
\label{SEC:5}

In this section we briefly study, from the point of view of our setting, the Perazzo cubic $V(f)\subset \bP^4$ and, in particular, the standard Artinian Gorenstein algebra $R=Q/\Ann_Q(f)$. For this section we will set $\bK=\bC$.
\vspace{2mm}

The Perazzo cubic (introduced by Perazzo in \cite{Per}) is the cubic threefold $X=V(f)$ with 
$$f=x_0x_3^2+2x_1x_3x_4+x_2x_4^2\in S=\bK[x_0,x_1,x_2,x_3]$$
and it is the "simplest" counterexample to Hesse's conjecture: up to projective transformations, it is the only cubic threefold with vanishing hessian in $\bP^4$ which is not a cone. This follows from the work of several authors which obtain a classification of the hypersurfaces in $\bP^4$ with vanishing hessian that are not cones. A comprehensive treatment of this problem can be found in \cite[Chapter 7.4]{Rus16} whereas the original articles dealing with this classification problem (also in higher dimension) are \cite{GN,Per,Fra54,Per57,Per64,Los04,CRS08,GR09}. 
\vspace{2mm}

Fix the notations as in Example \ref{EX:SAGAviaANN} with $n=4$ and let $f$ be the above cubic form. Then
$$R=Q/\Ann_Q(f)=R^0\oplus R^1\oplus R^2\oplus R^3$$ 
is a SAGA with socle in degree $N=3$. As recalled in Section \ref{SEC:3}, since $X$ is not a cone and its hessian vanishes, $R$ has codimension $5$ and does not satisfy $SLP$ (and $WLP$ as well). 
\vspace{2mm}

One has 
$$(\Ann_Q(F))_2=\langle y_0^2,y_0y_1,y_0y_2,y_0y_4,y_1^2,y_1y_2,y_2^2,y_2y_3,y_0y_3-y_1y_4,y_1y_3-y_2y_4\rangle\simeq \bK^{10}$$
and that $\{y_0y_3^2,y_1y_3y_4,y_2y_4^2\}$ are the only monomials of degree $3$ which are not $0$ in $Q$. More precisely, using the above relations, one has $R^N=\langle \sigma\rangle$ where $\sigma=y_0y_3^2=y_1y_3y_4=y_2y_4^2$. 
From these relations one has that
$$B_1=\{b_i\}_{i=1}^5=\{y_0,y_1,y_2,y_3,y_4\}\quad \mbox{ and }\quad B_2=\{c_i\}_{i=1}^5=\{y_3^2,y_3y_4,y_4^2,y_0y_3,y_2y_4\}$$
are basis for $R^1$ and $R^2$ respectively. Moreover, it is easy to check that $b_i\cdot c_j=\delta_{ij}\sigma$ so that $B_2$ is the dual basis of $B_1$ (by choosing the isomorphism $\bK\to R^N$ such that $1\mapsto\sigma$). Denote by $\{w_i\}_{i=1}^5$ and by $\{z_i\}_{i=1}^5$ the coordinates induced by $B_1$ and $B_2$ on $R^1$ and $R^2$  respectively and by $\tau$ the involution $\tau([x],[y])=([y],[x])$. With these notations, we have
that $\Gamma_{N-2}=\Gamma_{1}=\{([x],[y])\,|\, xy=0\}$ has $3$ irreducible components, $\Gamma,\tau(\Gamma)$ and $\Lambda$, all of dimension $4$ (one can show that we always have at least $3$ components under these assumption). Using coordinates $w_{1i}$ and $w_{2i}$ on the two factors of $\bP(R^1)\times \bP(R^1)$, we have
$$\Gamma=V(w_{13}w_{20}+w_{14}w_{21},w_{13}w_{21}+w_{14}w_{22}, w_{21}^2-w_{20}w_{22},w_{23},w_{24})\, \mbox{ and }\, \Lambda=V(w_{13},w_{14},w_{23},w_{24})$$
so
$Y=V(w_1^2-w_0w_2,w_3,w_4)$ is a conic. In particular, for $[y]\in Y$ general, we have $\dim(F_y))=3$. The morphism $\varphi(x)=x^2$ can be written in coordinates as
$$\underline{z}=\tilde{\varphi}(\underline{w})=(w_3^2,2w_3w_4,w_4^2,2(w_0w_3+w_1w_4),2(w_1w3+w_2w_4))$$
so $Y_2=V(w_3,w_4)\simeq \bP^2$ is the plane containing the conic $Y$ - here we have taken the reduced structure - and $Z=V(4z_0z_2-z_1^2)$ is a cone over a conic with vertex the line $V(z_0,z_1,z_2)$. The polar map $\nabla_Z$ associate to $Z$ is
$$\nabla_Z: [\underline{z}]\mapsto [\underline{w}]=[4z_2:-2z_1:4z_0:0:0]$$
and has image $Y$. The Gordan-Noether map $\psi_g$ associated to $g=4z_0z_2-z_1^2$ can be written in coordinates as
$$\psi_g(\underline{w})=[2w_4^2:-2w_3w_4:2w_4^2:0:0],$$
it is defined outside $Y_2$ and it defines a rational map from $\bP^4$ to $Y$, as observed in Remark \ref{REM:ALPHAandGN}. Finally, one can check that 
$\Gamma\cap \tau(\Gamma)=Y\times Y$ (this is, again, something that holds more generally), and
$$\Lambda=Y_2\times Y_2\qquad \Gamma\cap \Lambda=Y_2\times Y\qquad \tau(\Gamma)\cap \Lambda=Y\times Y_2$$
so $\Gamma\cap \tau(\Gamma)\cap \Lambda = Y\times Y$.

\end{document}